\documentclass[12pt]{article}
\usepackage{a4wide}

\usepackage[a4paper, margin=2.3 cm]{geometry}
\usepackage{amsmath,amssymb,amsthm}
\usepackage{color}
\usepackage{parskip}
\usepackage{hyperref}
\usepackage{parskip}
\usepackage{setspace}
\usepackage{graphicx}

\newtheorem{theorem}{Theorem}[section]

\newtheorem{corollary}[theorem]{Corollary}
\newtheorem{lemma}[theorem]{Lemma}
\newtheorem{proposition}[theorem]{Proposition}

\newtheorem*{definition*}{Definition}

\begin{document}
\title{Dot-product sets and simplices over finite rings}

\author{ Nguyen Van The \thanks{University of Science, Vietnam National University - Hanoi, Email: nguyenvanthe\_t61@hus.edu.vn} \and Le Anh Vinh\thanks{Vietnam National University - Hanoi, Email: vinhla@vnu.edu.vn. Vietnam Institute of Educational Sciences. Email: vinhle@vnies.edu.vn}}
\date{}
\maketitle  
\begin{abstract}
In this paper, we study dot-product sets and $k$-simplices in $\mathbb{Z}_n^d$ for odd $n,$ where $\mathbb{Z}_n$ is the ring of residues modulo $n$. We show that if $E$ is sufficiently large then the dot-product set of $E$ covers the whole ring. In higher dimensional cases, if $E$ is sufficiently large then the set of simplices and the set of dot-product simplices determined by $E$, up to congurence, have positive densities.
\end{abstract}

\section{Introduction}

The Erd\H{o}s distinct distance problem asks for the minimal number of distinct distances determined by a finite point set in $\mathbb{R}^d$, $d\geq 2$. This problem in the Euclidean plane has been solved by Guth and Katz \cite{guth-katz}. They show that a set of $N$ points in $\mathbb{R}^2$ has at least $cN/\log N$ distinct distances. 

Let $\mathbb{F}_q$ denote a finite field with $q$ elements, where $q$ is an odd prime power. For $E \subset \mathbb{F}^d_q$ ($d \geq 2$), the finite
analogue of the Erd\H{o}s distinct distance problem is to determine the smallest
possible cardinality of the set
\[ \Delta (E) =\{\|x- y\|= (x_1 - y_1)^2 +
   \ldots + (x_d - y_d)^2 : x, y \in E\}
   \subset \mathbb{F}_q . \]
This problem was first studied by Bourgain, Katz, and Tao \cite{bkt}. They showed that if $q$ is a prime, $q \equiv 3$ (mod $4$), then for
every $\epsilon > 0$ and $E \subset \mathbb{F}^2_q$ with
$|E| \ll q^{2-\epsilon}$, there exists $\delta > 0$ such that
$| \Delta (E) | \gg |E|^{\frac{1}{2} + \delta}$. The relationship between
$\epsilon$ and $\delta$ in their arguments, however, is difficult to
determine and to go up to higher dimensional cases. Here and throughout, $X \gg Y$ means that there exists $C>0$ such that $X \geq CY$.

Using Fourier analytic methods, Iosevich and Rudnev \cite{ir}  showed that for any odd prime power $q$ and any set $E \subset
\mathbb{F}^d_q$ of cardinality $|E| \gg q^{d / 2}$, we have $| \Delta (E) | \gg \min \left\{ q, q^{\frac{d - 1}{2}}
   |E| \right\} .$ Iosevich and Rudnev reformulated the question in analogy with the Falconer distance problem: \textit{How large
does $E \subset \mathbb{F}_q^d$, $d \geq 2$, needed to be ensure that $\Delta(E)$ contains a positive proportion of the elements of $\mathbb{F}_q$?} The above result implies that if $|E | \gg q^{\frac{d+1}{2}}$, then $\Delta(E) = \mathbb{F}_q$. This matches with Falconer's result in Euclidean setting that for a set $E$ with Hausdorff dimension greater than $(d+1)/2$, the distance set of $E$ is of positive measure. Hart, Iosevich, Koh and Rudnev \cite{iosevich-trans} that the exponent $(d+1)/2$ is sharp in odd dimensions, at least in general fields. In even dimensions, it is still conjectured that the correct exponent is $d/2$.
Chapman et al. \cite{chapman} made the first improvement by showing that if $E \subset \mathbb{F}_q^2$ satisfies $|E| \geq q^{4/3}$ then $|\Delta(E)| \gg cq$. In a recent paper \cite{murphy}, Murphy et al. improved the exponent $4/3$ to $5/4$ in the case of prime fields.  

In \cite{cip}, Covert, Iosevich, and Pakianathan extended the Erd\H{o}s distinct distances problem to the setting of finite cyclic rings $\mathbb{Z}_{p^l} = \mathbb{Z}/p^l\mathbb{Z}$, where p is a fixed odd prime and $l\geq 2$. Precisely, they proved that if  $E \subset \mathbb{Z}_q^d$ of cardinality \[ |E| \gg r(r+1)q^{\frac{(2r-1)d}{2r}+\frac{1}{2r}},\] then the distance set determined by $E$ will cover all units in $\mathbb{Z}_{p^l}$. In \cite{cov}, Covert extended the problem the ring of residues modulo $n$ for an arbitrary odd $n$. Let $p$ be the smallest prime divisor of $n$ and $\tau(n)$ be the number of divisors of $n$, Covert showed that if $|E| \gg \frac{\tau(n)n^d}{p^{(d-2)/2}}$ then the distance set determined by $E$ will cover all elements of the ring.  

Let
$E \subset \mathbb{F}_q^d$, we define the dot-product set of $E$ as follow
\[ \Pi(E) := \{ x \cdot y : x, y \in E \} \subset \mathbb{F}_q,\]
where $x \cdot y = x_1y_1 + \ldots x_dy_d$. Similarly, we can ask a question for the dot-product set instead of the distance set: \textit{How large does $E$ need to
ensure that the dot-product set $\Pi(E)$ can cover the whole field or at least a positive proportion of the field?} Hart and Iosevich
\cite{hi1}, using exponential sums, showed that for the product set to cover the whole field, one can take $|E| > q^{(d +
1) / 2}$ for any $d \geq 2$. Covert, Iosevich, and Pakianathan \cite{cip} extended the problem to the setting of finite cyclic rings $\mathbb{Z}_{p^l}$. They proved that if $E \subset \mathbb{Z}_q^d$ of cardinality 
\[ |E| \gg rq^{\frac{(2r-1)d}{r}+\frac{1}{2r}}, \]
then the dot-product set covers all units in $\mathbb{Z}_q$. In \cite{vinh3}, the second listed author also studied this result over the ring of residues modulo $n$ for an arbitrary $n$. In this paper, we will further extend the problem to cover the whole ring. Note that, our result is in line with the result of Covert in \cite{cov} for the Erd\H{o}s distinct distances problem.
 
A classical result due to Furstenberg, Katznelson and Weiss \cite{fkw} states that
if $E \subset \mathbb{R}^2$ of positive upper Lebesgue density,
then for any $\delta > 0$, the $\delta$-neighborhood of $E$ contains
a congruent copy of a sufficiently large dilate of every three-point
configuration. In the case of $k$-simplex, ussing Fourier analytic techniques, Bourgain \cite{b1} showed that a set $E$ of positive
upper Lebesgue density always contains a sufficiently large dilate of every
non-degenerate $k$-point configuration where $k < d$. 
Hart and Iosevich \cite{hi} were the first to study an analog of this
question in finite field geometries. Let $P_k$ and $P'_k$ be two $k$-simplices in vector space $\mathbb{}_q^d$. We say that $P_k \sim P_k^{'}$ if there exist $\tau \in
\mathbb{F}_q^d$, and $O \in SO_d (\mathbb{F}_q)$, the set of $d$-by-$d$
orthogonal matrices over $\mathbb{F}_q$, such that $P^{'}_k = O(P_k) + \tau$.
Hart and Iosevich \cite{hi} observed that, under this equivalent relation, one may
specify a simplex by the distances determined by its vertices. When $d \ge \binom{k+1}{2}$, they showed that if $E \subset \mathbb{F}_q^d$ ($d \geq \binom{k + 1}{2}$) of
cardinality $|E| \gg C q^{\frac{k d}{k + 1} + \frac{k}{2}}$ then
$E$ contains a congruent copy of every $k$-simplices with the exception
of simplices with zero distances. Using spectral graph theory, this lower bound on the set size was improved to $|E| \gg
q^{\frac{d - 1}{2} + k}$ by the second listed author
\cite{vinh-pg} for the case of $d \geq 2 k$.

If we only want to cover a positive proportion of all possible simplices, the above bounds can be further improved. Chapman et al \cite{chapman} showed that if $|E|
\gtrsim q^{\frac{d + k}{2}}$ ($d \geq k$) then the set of $k$-simplices, up to congruence, has density greater than $c$. 
Using group action approach,  Bennett et al. \cite{bennett} proved that if $E \gg q^{d - \frac{d-1}{k+1}}$ then $E$  determines a positive proportion of all $k$-simplices. In \cite{hiep}, H. Pham, T. Pham and the second listed author gave an improvement of this result in the case when $E$ is the Cartesian product of sets.

In line with the study of simplices in vector spaces over finite fields,
the second listed author \cite{vinh-ajm} also studied the distribution of simplices with respect to the dot-prouct. Note that, this problem can be viewed as the solvability of systems of bilinear equations over finite fields. 

In this paper, we study analogue results of $k$-simplices and dot-product $k$-simplices in $\mathbb{Z}_n^d$ for an arbitrary odd integer $n$. We will show that any sufficient large subset $E \subset \mathbb{Z}_n^d$, the set of $k$-simplices and the set of dot-product $k$-simplices determined by $E$, up to
congruence, have positive densities.

\section{Statements of results}

\subsection{Dot-product sets}

Let $\mathbb{Z}_n$ be the ring of residues mod $n$ where $n$ is a large odd integer. Denote $\mathbb{Z}_n^{\times}$ be the set of units in $\mathbb{Z}_n.$ The finite Euclidean space $\mathbb{Z}_n^d$ consists of column vectors $x,$ with $i^{th}$ entry $x_i \in \mathbb{Z}_n.$ For a subset $E \in \mathbb{Z}_n$, we define the \textit{dot-product set} of $E$ as follows 
\[ \Pi(E) := \left\{x \cdot y: x,y \in E\right\} \subset \mathbb{Z}_n,\]
where 
\[   x \cdot y = x_1y_1 + \dots x_dy_d \]
is the usual dot product.  
Using Fourier analysis, Covert, Iosevich, and Pakianathan \cite{cip} showed that if the set $E$ is large enough, its product set will cover all units in $\mathbb{Z}_n$. 

\begin{theorem}\label{theo1}\textsl{(Covert, Iosevich, Pakianathan, \cite{cip})}
Let $E \subset \mathbb{Z}_q^d,$ where $\ell \ge 2$ and $q=p^{\ell}$ is an odd prime power. Suppose that $|E| \gg \ell q^{\frac{(2\ell - 1)d}{2\ell} + \frac{1}{2\ell}}.$ We have
\[ \Pi(E) \supset \mathbb{Z}_q^{\times}.\] 
\end{theorem}
In \cite{cip},  it was shown that Theorem \ref{theo1} is close to optimal in the sense that there exist a value $b = b(p)$ and a subset $E \subset \mathbb{Z}_q^d$ with cardinality $|E| = b q^{\left(\frac{2\ell - 1}{2\ell}\right)d}$ such that $\Pi(E) \cap \mathbb{Z}^{\times}_q = \emptyset.$ For such constructed set $E,$ we have $\left|\Pi(E)\right| \leq p^{\ell - 1} = \underline{o}(q).$   

In the general case of the ring of residues modulo $n$ with $n$ is a large odd integer, the second listed author obtained the following result (\cite{vinh3}).
\begin{theorem} \label{theo3}\textsl{(Vinh, \cite{vinh3})}
Let $n$ be a large odd integer. Denote $\gamma(n)$ be the smallest prime divisor of $n$ and $\tau(n)$ be the number of divisors of $n.$ 
\begin{itemize}
\item[a)] Suppose that $E \subset \mathbb{Z}_n^d$ with cardinality \[|E| \ge \frac{\sqrt{2} \tau(n)n^{d}}{\gamma(n)^{(d-1)/2}}.\]
Then, we have $\mathbb{Z}_n^{\times} \subset \Pi(E)$.
\item[b)] Suppose that $E \subset \mathbb{Z}_n^d$ with cardinality \[|E| \ge \frac{2 \sqrt{\tau(n)}n^{d+1}}{\gamma(n)^{d/2}}.\]
Then, we have $\Pi(E)  = \mathbb{Z}_n$.
\end{itemize}
\end{theorem}
The first part of Theorem \ref{theo3} is a generalization of Theorem \ref{theo1}. In the second part, to cover the whole ring $\mathbb{Z}_n,$ we need a weaker condition on the sizes of $E.$ On the other hand, the result in the second part of Theorem \ref{theo3} is trivial when $n \ge \gamma(n)^{d/2}.$ More precisely, in the case of finite cyclic rings $\mathbb{Z}_{p^{\ell}}$, the result is non-trivial only if $\ell < d/2$.

It is of interest to extend Theorem \ref{theo1} and Theorem \ref{theo3} to cover the whole ring $\mathbb{Z}_n.$ Our first result is the following.
\begin{theorem}\label{main.result}
Let $E \subset \mathbb{Z}_n^d$ where $d >2$ and $n$ is a large odd integer. Suppose that 
\[ |E| > \dfrac{\tau(n)n^d }{\gamma(n)^{(d-2)/2}}.\]
Then, we have $\Pi(E) = \mathbb{Z}_q$.
\end{theorem}
Note that, this result improves the second part of Theorem \ref{theo3} and aligns with Covert's result \cite{cov} for Erd\H{o}s distance problem. Besides, let $E \subset \mathbb{Z}_n^d$ be the set of all elements in $\mathbb{Z}_n^d$ withs all complements are divisible by $\gamma(n),$ then $|E| = n^d \gamma(n)^{-d}$ and $\Pi(E)$ contains no non-unit element of $\mathbb{Z}_n$. It shows that Theorem \ref{main.result} is asymptotically sharp as we fix $\gamma(n)$ and $d$ then let $n $ goes to infinity. Moreover, this result is non-trivial for $d \ge 3$ if $\tau(n) = \underline{o}(n^\varepsilon)$ for all $\varepsilon > 0$.

As a direct consequence, we has the following corollary in the case $n = p^{\ell}.$
\begin{corollary} 
Let $E \subset \mathbb{Z}_{q}^d,$ where $q = p^{\ell}$ and $d \ge 3.$ Suppose that $|E| > (\ell +1)q^{\frac{(2\ell - 1)d}{2\ell}+\frac{1}{\ell}}.$ Then, we have $\Pi(E) =  \mathbb{Z}_q$.
\end{corollary}

\subsection{Distribution of $k$ simplices}
Since a geometric justification of the notion of distance is that an orthogonal transformation preserves this distance, a $k$-simplex in a subset $ E \subset \mathbb{Z}_n^d$ can be definited recursively by setting 
\[ \mathcal{T}_{l_k} = \left\{ (\mathbf{x}_0,\dots,\mathbf{x}_{k-1},\mathbf{x}_k) \in \mathcal{T}_{l_{k-1}} \times E: \Vert \mathbf{x}_i - \mathbf{x}_k \Vert = t_{i,k}, \, i = 0,\dots,k-1\right\}, \]
in which $l_k = l_{k-1} \cup \left\{ (t_{0,k}, \dots, t_{k-1,k}), \, t_{i,j} \in \mathbb{Z}_n\right\}$ and 
\[ \mathcal{T}_{l_1} = \left\{ (\mathbf{x}_0, \mathbf{x}_1) \in E^2\,:\, \Vert \mathbf{x}_0 - \mathbf{x}_1 \Vert = t_{0,1} \right\}.\]
Denote $\mathcal{T}_k(E) := \left\{l_k: |\mathcal{T}_{l_k}|>0\right\}$ be the set of $k$-simplices determined by $E.$ We have the following the result.
\begin{theorem}\label{main.k-simplices}
Let $E \subset \mathbb{Z}_n^d.$ Suppose that 
\[ |E| \gg \dfrac{\sqrt{\tau(n)}n^{d+\frac{k-1}{2}}}{\gamma(n)^{(d-1)/2}}\]
with $k \leq d,$ then $E$ determines a positive proportion of all $k-$simplices over $\mathbb{Z}_n^d.$ In other words, 
\[ \left|\mathcal{T}_k(E)\right| \gg n^{\binom{k+1}{2}}.\] 
\end{theorem}

Similarly, one can define a $k$-simplex with dot-product instead of distance function. A dot-product $k$-simplex in a subset $E \subset \mathbb{Z}_n^d$ can be defined recursively by setting 
\[ \mathcal{P}_{l_k} = \left\{ (\mathbf{x}_0,\dots,\mathbf{x}_{k-1},\mathbf{x}_k) \in \mathcal{P}_{l_{k-1}} \times E: \mathbf{x}_i \cdot\mathbf{x}_k = t_{i,k}, \, i = 0,\dots,k-1\right\}, \]
in which $l_k = l_{k-1} \cup \left\{ (t_{0,k}, \dots, t_{k-1,k}), \, t_{i,j} \in \mathbb{Z}_n\right\}$ and 
\[ \mathcal{P}_{l_1} = \left\{ (\mathbf{x}_0, \mathbf{x}_1) \in E^2\,:\, \mathbf{x}_0 \cdot \mathbf{x}_1 = t_{0,1} \right\}.\]
Denote $\mathcal{P}_k(E) := \left\{l_k: |\mathcal{P}_{l_k}|>0\right\}$ be the set of dot-product $k$-simplices determined by $E.$ We have the following the result. 
\begin{theorem}\label{main.k-simplices dot product}
Let $E \subset \mathbb{Z}_n^d.$ Suppose that 
\[ |E| \gg \dfrac{\sqrt{\tau(n)}n^{d+\frac{k-1}{2}}}{\gamma(n)^{(d-1)/2}}\]
with $k \leq d$, Then, $E$ determines a positive proportion of all dot-product $k-$simplices over $\mathbb{Z}_n^d$. In other words,  
\[ \left|\mathcal{P}_k(E)\right| \gg n^{\binom{k+1}{2}}.\] 
\end{theorem}

\section{Dot-product sets - Proof of Theorem \ref{main.result}}

We first recall some basic results on Fourier Analysis in $\mathbb{Z}_n^d$. For $f : \mathbb{Z}_n^d \longrightarrow \mathbb{C},$ we define the Fourier transform of $f$ as 
\[ \widehat{f}(m) = n^{-d}\sum_{x \in \mathbb{Z}_n^d} f(x)\chi(- x \cdot m),\]
where $\chi(x) = \exp(2\pi ix/n).$ Since $\chi$ is a character on the additive group $\mathbb{Z}_n,$ we have the following orthogonality property. 

\begin{lemma}\label{ortho-property}
We have 
\begin{align*}
n^{-d} \sum_{x \in \mathbb{Z}^d_n} \chi(x \cdot m) = 
\begin{cases}
1 & m = (0,\dots,0) \\
0 & \text{otherwise}
\end{cases}
\end{align*}
\end{lemma}

The Plancherel and inversion-like identities can be derived from Lemma \ref{ortho-property}.
\begin{proposition}
Let $f$ and $g$ be complex-valued functions defined on $\mathbb{Z}_n^d$. Then, 
\begin{equation}\label{Inversion-Eq}
f(x) = \sum_{ m \in \mathbb{Z}_n^d} \chi(x \cdot m) \widehat{f}(m)
\end{equation}
\begin{equation}\label{Plancherel-Eq}
n^{-d} \sum_{x \in \mathbb{Z}_n^d} f(x)\overline{g(x)} = \sum_{m \in \mathbb{Z}_n^d} \widehat{f}(m) \overline{\widehat{g}(m)}
\end{equation}
\end{proposition}

We are now ready to give a proof of Theorem \ref{main.result}. We will follow a similar approach as in \cite{cip}.

\begin{proof}[\bf Proof of Theorem \ref{main.result}]
Without loss of generality, we suppose that $n$ has the prime decomposition $n = p_1^{\alpha_1}p_2^{\alpha_2}\dots p_k^{\alpha_k},$ where $2 < p_1 < p_2 < \dots < p_k$ and $\alpha_i > 0$ for each $i = 1,2,\dots,k.$ 
For $E \subset \mathbb{Z}_n^d,$ we define the incidence function 
\[ \mu(t) = \left\{(x,y) \in E \times E: x \cdot y = t\right\}.\]
In order to show that $\Pi(E) = \mathbb{Z}_n$, we will demonstrate that $\mu(t) > 0.$ We rewrite
\begin{align}\label{eq1}
\mu(t) &= n^{-1} \sum_{s \in \mathbb{Z}_n} \sum_{x,y \in E} \chi\left(s (x\cdot y)\right) \chi(-st) \notag\\
&= n^{-1}|E|^2 + \mathcal{M}(t), 
\end{align} 
where 
\begin{align*}
\mathcal{M}(t) &= n^{-1}\sum_{s \neq 0} \sum_{x,y \in E} \chi\left(s(x \cdot y)\right) \chi(-st). 
\end{align*}
Define 
\[ val(s) := (val_{p_1}(s),\dots,val_{p_k}(s))\]
where $val_{p_i}(x) = r$ if $p_i^r | x$ but $p_i^{r+1} \nmid x.$ 
For each $s \neq 0,$ we write $s = p_1^{\beta_1}\dots p_k^{\beta_k} \overline{s}$ where $\overline{s} \in \mathbb{Z}_{n'}^{\times}$ is uniquely determined for $n'=p_1^{\alpha_1 - \beta_1} \dots p_k^{\alpha_k - \beta_k}$ and $\beta_i \ge 0.$ Note that, $s \neq 0$ so $\beta_i < \alpha_i$ for some $i$. We will use the notation $\sum_{\beta}$ to denote the sum over all such $(\beta_1,\dots,\beta_k)$'s. 

Using the notation as above, we can write  $\mathcal{M}(t) = \sum_{\beta} \mu_{\beta}(t),$ where
\[ \mu_{\beta}(t) = q^{-1} \sum_{\substack{ s \in \mathbb{Z}_n \setminus \{0\}: val(s) = \beta}} \sum_{x,y \in E} \chi\left(s(x \cdot y)\right) \chi\left(-st\right).\]  

We will find an upper bound of $\mu_{\beta}(t)$ for each $\beta=(\beta_1,\dots,\beta_n).$ Indeed, viewing the term $\mu_{\beta}(t)$ as a sum in $x$-variable, applying Cauchy-Shwarz inequality, then extending the sum over $x \in E$ to the sum over $x \in \mathbb{Z}_n^d,$ we see that 
\begin{align*}
|\mu_{\beta}(t)|^2 &\leq |E|n^{-2} \sum_{x \in \mathbb{Z}^d_n} \sum_{y,y' \in E} \sum_{s,s' \in \mathbb{Z}^{\times}_{n'}} \chi\left(p_1^{\beta_1}\dots p_k^{\beta_k}(sy-s'y')x\right)\chi\left(p_1^{\beta_1}\dots p_k^{\beta_k}t(s'-s)\right) \\
&\leq |E|n^{d-2} \sum_{s,s' \in \mathbb{Z}^{\times}_{n'}} \sum_{\substack{y,y' \in E: \\ p_1^{\beta_1}\dots p_k^{\beta_k}(sy - s'y')=\mathbf{0}}} \chi\left(p_1^{\beta_1}\dots p_k^{\beta_k}t(s'-s)\right)\\
&= |E|n^{d-2} \sum_{a,b \in \mathbb{Z}^{\times}_{n'}} \sum_{\substack{y,y' \in E: \\ p_1^{\beta_1}\dots p_k^{\beta_k}\left(b(ay - y')\right)=\mathbf{0}}} \chi\left(p_1^{\beta_1}\dots p_k^{\beta_k}t\left(b(1-a)\right)\right).
\end{align*}

Since $b$ is a unit in $\mathbb{Z}_{n'},$ we have $ay-y'=\mathbf{0}$ in $\mathbb{Z}_{n'}^d.$ This implies that
\begin{align*}
|\mu_\beta(t)|^2 &\leq |E|n^{d-2} \sum_{a, b \in \mathbb{Z}^{\times}_{n'}} \sum_{\substack{y,y' \in \mathbb{Z}_n^d: \\ p_1^{\beta_1}\dots p_k^{\beta_k}(b(ay-y'))=\mathbf{0}}} E(y)E(y')\chi\left(p_1^{\beta_1}\dots p_k^{\beta_k}t\left(b(1-a)\right)\right) \\
&\leq |E|n^{d-2} \sum_{a, b \in \mathbb{Z}^{\times}_{n'}} \sum_{\substack{y,y' \in \mathbb{Z}_n^d: \\ p_1^{\beta_1}\dots p_k^{\beta_k}(b(ay-y'))=\mathbf{0}}} \left|E(y)E(y')\chi\left(p_1^{\beta_1}\dots p_k^{\beta_k} t\left(b(1-a)\right)\right)\right| \\
&= |E|n^{d-2} \sum_{a, b \in \mathbb{Z}^{\times}_{n'}} \sum_{y \in \mathbb{Z}_n^d}  |E(y)|\left|R(ay)\right|,
\end{align*}
where $R(\gamma) = \left\{ y' \in E: y' \equiv \gamma \, \left(\mathrm{mod}\, n'\right)\right\}.$ Since the Kernel of the canonical projection $K$ from $\mathbb{Z}_n^d$ to $\mathbb{Z}^d_{n'}$ defined by
\[ K : \quad y \mapsto y \,\,\mbox{mod}\,\,n',\] 
has size of $(n/n')^{d},$ we have
\begin{align*}
|\mu_\beta(t)|^2 &\leq |E|n^{d-2} \sum_{a, b \in \mathbb{Z}^{\times}_{n'}} \sum_{y \in \mathbb{Z}_n^d}  |E(y)|\left|R(ay)\right| \\
&= |E|n^{d-2} \sum_{a,b \in \mathbb{Z}_{n'}^{\times}} \sum_{y \in \mathbb{Z}_n^d}  \left(\frac{n}{n'}\right)^d E(y) \\
&\leq |E|^2 \frac{n^{2d-2}}{n'^{d}} \left|\mathbb{Z}^{\times}_{n'}\right|^2  \\
&\leq |E|^2 n^{2d-2} n'^{2-d}.
\end{align*}
On the other hand, we know that $n' = p_1^{\alpha_1-\beta_1} \dots p_k^{\alpha_k - \beta_k} \ge p_1 = \gamma(n)$ since $p_1 < p_2 < \dots < p_k$ and $\beta_i < \alpha_i$ for some $i.$ Hence,
\begin{align*}
|\mu_{\beta}(t)| &\leq |E| n^{d-1} p_1^{-\frac{d-2}{2}} =\dfrac{n^{d-1}|E|}{\gamma(n)^{(d-2)/2}}. 
\end{align*}
Therefore, we have 
\[|\mathcal{M}(t)| \leq \sum_{\beta} \left|\mu_{\beta}(t)\right| \leq \dfrac{\tau(n)n^{d-1}|E|}{\gamma(n)^{(d-2)/2}}.\]
It follows from \eqref{eq1} that $\mu(t) > 0$ whenever 
\[ |E| > \dfrac{\tau(n)n^d}{\gamma(n)^{(d-2)/2}}.\]
This completes the proof of Theorem \ref{main.result}.
\end{proof}

\section{Distribution of simplices - Proof of Theorem \ref{main.k-simplices}}
\subsection{Counting number of $k$-stars}

Define the $k$-star set determined by a base of $k$ points $y^1,\dots,y^k \in E$ as follows 
\[ \Delta_{y^1,y^2,\ldots,y^k}(E) = \left\{\left(\Vert x - y^1 \Vert,\ldots, \Vert x - y^k \Vert\right) \in \mathbb{Z}_q^k: x \in E \right\}.\]
The main result of this section is to count the number of $k$-stars with bases in a point set $E$. 

\begin{theorem}\label{k-star}
Let $E \subset \mathbb{Z}_n^d$ with $n \ge 3$ be an odd integer. Suppose that 
\[ |E| \gg \dfrac{\sqrt{\tau(n)}n^{d+\frac{k-1}{2}}}{\gamma(n)^{(d-1)/2}}.\]
Then, we have 
\[ \frac{1}{|E|^k} \sum_{y^1, \dots , y^k \in E} \left| \Delta_{y^1,\dots,y^k}(E) \right| \gg n^{k}.\]
\end{theorem}

For $t_1,\dots,t_k \in \mathbb{Z}_n$ and $E \subset \mathbb{Z}_n^d,$ we define the counting function
\begin{align*}
 \nu_{y^1,\dots,y^k}(t_1,\dots,t_k) &:= \left|\left\{ x \in E \,:\, \Vert x - y^i \Vert = t_i,\, \forall i = 1, \dots ,k \right\}\right|\\
&= \sum_{ \Vert x - y^1\Vert = t_1, \dots, \Vert x - y^k \Vert = t_k} E(x).
\end{align*}
The following lemma plays an significant role in proof of Theorem \ref{k-star}. 

\begin{lemma}\label{lemma.k-star} Let $E \subset \mathbb{Z}_n^d$ where $n \ge 3$ is an odd integer. Then 
\[ \mathcal{M}_{k} = \sum_{y^1,\dots,y^k \in E} \sum_{t_1,\dots,t_k \in \mathbb{Z}_n} \nu^2_{y^1,\dots,y^k}(t_1,\dots,t_k) \ll \dfrac{|E|^{k+2}}{n^k} + \frac{\tau(n)n^{2d-1}}{\gamma(n)^{d-1}}|E|^k.\]
\end{lemma}

\begin{proof}
We proceed by induction on $k$. For the initial case $k = 1,$ we use the notation $\nu_y(t)$ instead of $\nu_{y^1}(t_1)$ for the counting function. We have
\[ \nu_{y}(t)^2 = \sum_{\Vert x - y \Vert = \Vert x' - y\Vert = t} E(x) E(x').\] 
Summing in $y \in E$ and $t \in \mathbb{Z}_n,$ then applying Lemma \ref{ortho-property}, we have 
\begin{align*}
 \sum_{ y \in E} \sum_{t \in \mathbb{Z}_n} \nu_y(t)^2 &= \sum_{\Vert x - y \Vert = \Vert x' - y \Vert} E(x)E(x')E(y)  \\
 &= n^{-1} \sum_{s \in \mathbb{Z}_n} \sum_{y,x,x' \in \mathbb{Z}^d_n} \chi\left( s\left( \Vert x - y \Vert - \Vert x' - y \Vert\right)\right) E(y)E(x)E(x') \\
 &= n^{-1}|E|^3 + n^{-1} \sum_{s \neq 0}\sum_{y,x,x' \in \mathbb{Z}^d_n} \chi\left( s\left( \Vert x - y \Vert - \Vert x' - y \Vert\right)\right) E(y)E(x)E(x') \\
 &= n^{-1}|E|^3 + \mathcal{R}.  
\end{align*}
Since $\Vert x - y \Vert - \Vert x' - y \Vert = \left( \Vert x \Vert - 2y \cdot x\right) - \left( \Vert x' \Vert - 2y \cdot x'\right),$ we can rewrite $\mathcal{R}$ as 
\begin{align*}
 \mathcal{R} &= n^{-1} \sum_{ s \neq 0} \sum_{x,x',y \in \mathbb{Z}_n^d} \chi\left(\Vert x \Vert - 2y \cdot x\right)\chi\left(2y \cdot x' - \Vert x' \Vert\right) E(x)E(x')E(y') \\
 &= n^{-1} \sum_{s \neq 0} \sum_{ y \in E} \left| \sum_{x \in E} \chi\left(s\left(\Vert x \Vert - 2y \cdot x \right)\right) \right|^2. 
\end{align*}
It follows that $\mathcal{R} \ge 0$ and 
\begin{align*}
\mathcal{R} &\leq n^{-1} \sum_{s \neq 0} \sum_{y \in \mathbb{Z}^d_n}\left| \sum_{x \in E} \chi\left(s\left(\Vert x \Vert - 2y \cdot x \right)\right) \right|^2 \\
&= n^{-1} \sum_{s \neq 0} \sum_{y \in \mathbb{Z}^d_n}\sum_{x,x' \in E} \chi\left(s\left(\Vert x \Vert - \Vert x' \Vert \right)\right) \chi\left(-2sy\cdot\left(x - x' \right)\right).
\end{align*}

Without loss of generality, we suppose that $n = p_1^{\alpha_1} \dots p_{\ell}^{\alpha_\ell}.$ Define
\[ val(s) := (val_{p_1}(s),\dots,val_{p_\ell}(s))\]
where $val_{p_i}(x) = r$ if $p_i^r | x$ but $p_i^{r+1} \nmid x.$ 
For each $s \neq 0,$ we write $s = p_1^{\beta_1}\dots p_{\ell}^{\beta_{\ell}} \overline{s}$ where $\overline{s} \in \mathbb{Z}_{n'}^{\times}$ is uniquely determined for $n'_{\beta}=p_1^{\alpha_1 - \beta_1} \dots p_{\ell}^{\alpha_\ell - \beta_\ell}$ and $\beta_i \ge 0.$ Since $s \neq 0$, $\beta_i < \alpha_i$ for some $i$. We will use the notation $\sum_{\beta}$ to denote the sum over all such $(\beta_1,\dots,\beta_\ell)$'s. 

Using this notation, we have  $\mathcal{R} \leq \sum_{\beta} \mathcal{R}_{\beta},$ where
\[ \mathcal{R}_{\beta} = n^{-1} \sum_{\substack{ s \in \mathbb{Z}_n \setminus \{0\}: \\ val(s) = \beta}}  \sum_{y \in \mathbb{Z}^d_n}\sum_{x,x' \in E} \chi\left(s\left(\Vert x \Vert - \Vert x' \Vert \right)\right) \chi\left(-2sy\cdot\left(x - x' \right)\right). \]  

For each $\beta = (\beta_1, \dots,\beta_{\ell}),$ denote $n'_\beta = p_1^{\alpha_1 - \beta_1} \dots p_{\ell}^{\alpha_\ell - \beta_\ell}$ and $n_\beta = p_1^{\beta_1},\dots,p_{\ell}^{\beta_\ell}.$ Now, we will bound $\mathcal{R}_\beta.$ Applying the orthogonality property (Lemma \ref{ortho-property}), we have 
\begin{align*}
\mathcal{R}_{\beta} &= n^{-1} \sum_{\overline{s} \in \mathbb{Z}_{n'_\beta}^{\times}} \sum_{y \in \mathbb{Z}_n^d} \sum_{x,x' \in E} \chi\left(p_1^{\beta_1}\dots p_{\ell}^{\beta_\ell}\overline{s}\left(\Vert x \Vert - \Vert x' \Vert \right)\right) \chi\left(-2p_1^{\beta_1}\dots p_{\ell}^{\beta_\ell}\overline{s} y \cdot \left( x - x'\right)\right)\\
 &= n^{-1} \sum_{\overline{s} \in \mathbb{Z}^{\times}_{n'_\beta}}  \sum_{y \in \mathbb{Z}^d_n}\sum_{\substack{x,x' \in E: \\ p_1^{\beta_1}\dots p_{\ell}^{\beta_\ell}(x-x') = \mathbf{0} }} \chi\left(p_1^{\beta_1}\dots p_{\ell}^{\beta_\ell}\overline{s}\left(\Vert x \Vert - \Vert x' \Vert \right)\right) \\
&\leq n^{d - 1}\sum_{\overline{s} \in \mathbb{Z}_{n'}^{\times}} \left| \left\{ x,x' \in E: p_1^{\beta_1}\dots p_{\ell}^{\beta_\ell}(x-x')=\mathbf{0}\right\}\right| \\
&< n^{d-1} n'_\beta \sum_{x' \in E} \left| \left\{ x \in E: p_1^{\beta_1}\dots p_{\ell}^{\beta_\ell}(x-x')=\mathbf{0}\right\}\right|.
\end{align*}
On the other hand, since the Kernel of the canonical projection $K$ from $\mathbb{Z}_n^d$ to $\mathbb{Z}^d_{n'_{\beta}}$ defined by
\[ K : \quad y \mapsto y \,\,\mbox{mod}\, n'_\beta,\] 
has the size of $(n/n'_\beta)^{d},$ for each $x' \in E,$ there exist $(n/n'_{\beta})^{d} = n_{\beta}^d$ solutions to the equation $n_\beta (x-x') = p_1^{\beta_1}\dots p_{\ell}^{\beta_\ell} (x - x') = 0.$ Therefore, we obtain
\begin{equation} \label{eq.solutions}
 \left| \left\{ x \in E: p_1^{\beta_1} \dots p_{\ell}^{\beta_\ell}(x-x') = \mathbf{0}\right\} \right| \leq \left(\frac{n}{n'_\beta}\right)^d.
\end{equation}
It follows that
\begin{align*}
\mathcal{R}_\beta < n^{d-1}n'_\beta \sum_{x' \in E} \left(\dfrac{n}{n'_\beta}\right)^d = \dfrac{n^{2d-1}}{n'^{d-1}_{\beta}}|E| \leq \dfrac{n^{2d-1}}{\gamma(n)^{d-1}}|E|.
\end{align*}
Putting all together, we have 
\begin{align*}
\sum_{y \in E} \sum_{t \in \mathbb{Z}_n} \nu_y(t)^2 &< n^{-1}|E|^3 + \sum_{\beta} \mathcal{R}_{\beta} \\
&\leq  n^{-1}|E|^3 +  \frac{\tau(n)n^{2d-1}}{\gamma(n)^{d-1}}|E|.
\end{align*}
This completes the proof of the initial case. 

Now, suppose that the statement holds for $k-1$
\[ \sum_{y^1,\dots,y^{k-1} \in E} \sum_{t_1,\dots,t_{k-1} \in \mathbb{Z}_n} \nu^2_{y^1,\dots,y^{k-1}}(t_1,\dots, t_{k-1}) \ll \frac{|E|^{k+1}}{n^{k-1}} + \dfrac{\tau(n) n^{2d-1}}{\gamma(n)^{d-1}}|E|^{k-1}. \] 
We will show that the statement holds for $k$. We have 
\begin{align*}
 \sum_{ y^1,\dots,y^{k-1},y^k \in E} \sum_{t_1,\dots,t_{k-1}, t_k \in \mathbb{Z}_n} \nu^2_{y^1,\dots,y^{k-1},y^k}(t_1,\dots,t_{k-1},t_k) &=
\\   
\underset{
\begin{subarray}{c}
\Vert x - y^{1}\Vert = \Vert x' - y^{1} \Vert, \dots, \Vert x - y^{k}\Vert = \Vert x' - y^{k} \Vert 
\end{subarray}}{\sum \cdots \sum}& E(y^1)\dots E(y^k) E(x) E(x').
\end{align*}
Applying the orthogonality property (Lemma \ref{ortho-property}), we obtain 
\begin{align*}
 \sum_{ y^1,\dots,y^{k-1},y^k \in E} \sum_{t_1,\dots,t_{k-1}, t_k \in \mathbb{Z}_n} \nu^2_{y^1,\dots,y^{k-1},y^k}(t_1,\dots,t_{k-1},t_k) &=\\
n^{-1} \sum_{ \substack{ s \in \mathbb{Z}_n,\\ x,x',y^1,\dots, y^k \in E}} 
\underset{
\begin{subarray}{c}
\Vert x - y^{1}\Vert = \Vert x' - y^{1} \Vert, \dots, \Vert x - y^{k-1}\Vert = \Vert x' - y^{k-1} \Vert 
\end{subarray}}{\sum \cdots \sum}  
&\chi\left( s\left(\Vert x \Vert - 2y^k \cdot x\right)\right)\chi\left(-s\left(\Vert x' \Vert - 2y^k \cdot x'\right)\right) 
\end{align*}
since 
\[\Vert x - y^k \Vert - \Vert x' - y^k \Vert = \left( \Vert x \Vert - 2y^k\cdot x\right) - \left( \Vert x' \Vert - 2y^k \cdot x'\right).\]
Separating the term $s = 0$ then applying the induction hypothesis, we have
\[ \mathcal{M}_k \ll \dfrac{|E|^{k+2}}{n^k} + \dfrac{\tau(n)n^{2d-2}}{\gamma(n)^{d-1}}|E|^k + \mathcal{N},\]
where
\begin{align*}
\mathcal{N} &= n^{-1} \sum_{ \substack{ s \neq 0,\\ x,x',y^1,\dots, y^{k-1} \in E \\ y^k \in E}} 
\underset{
\begin{subarray}{c}
\Vert x - y^{1}\Vert = \Vert x' - y^{1} \Vert, \dots, \Vert x - y^{k-1}\Vert = \Vert x' - y^{k-1} \Vert 
\end{subarray}}{\sum \cdots \sum}  
\chi\left( s\left(\Vert x \Vert - 2y^k \cdot x\right)\right)\chi\left(-s\left(\Vert x' \Vert - 2y^k \cdot x'\right)\right)  \\
&= n^{-1} \sum_{ \substack{ s \neq 0,\\ y^1,\dots, y^{k-1} \in E \\ y^k \in E}} \sum_{t_1,\dots,t_{k-1} \in \mathbb{Z}_n}\left|\sum_{x \in E}
\underset{
\begin{subarray}{c}
\Vert x - y^{1}\Vert = t_1, \dots, \Vert x - y^{k-1}\Vert = t_{k-1} 
\end{subarray}}{\sum \cdots \sum}  
 \chi\left( s\left(\Vert x \Vert - 2y^k \cdot x\right)\right)\right|^2 \\
&\leq n^{-1} \sum_{y^k \in \mathbb{Z}_n^d} \sum_{ \substack{ s \neq 0,\\ y^1,\dots, y^{k-1} \in E}} \sum_{t_1,\dots,t_{k-1} \in \mathbb{Z}_n}\left|\sum_{x \in E}
\underset{
\begin{subarray}{c}
\Vert x - y^{1}\Vert = t_1, \dots, \Vert x - y^{k-1}\Vert = t_{k-1} 
\end{subarray}}{\sum \cdots \sum}  
 \chi\left( s\left(\Vert x \Vert - 2y^k \cdot x\right)\right)\right|^2  \\
&= n^{-1} \sum_{y^k \in \mathbb{Z}_n^d} \sum_{ \substack{ s \neq 0,\\ y^1,\dots, y^{k-1} \in E}}  \underset{
\begin{subarray}{c}
\Vert x - y^{1}\Vert = \Vert x' - y^{1} \Vert, \dots, \Vert x - y^{k-1}\Vert = \Vert x' - y^{k-1} \Vert \\
x,x' \in E
\end{subarray}}{\sum \cdots \sum} \chi\left( s\left(\Vert x \Vert - 2y^k \cdot x\right)\right)\chi\left(-s\left(\Vert x' \Vert - 2y^k \cdot x'\right)\right) \\
&= n^{-1} \sum_{y^k \in \mathbb{Z}_n^d} \sum_{ \substack{ s \neq 0,\\ y^1,\dots, y^{k-1} \in E}}  \underset{
\begin{subarray}{c}
\Vert x - y^{1}\Vert = \Vert x' - y^{1} \Vert, \dots, \Vert x - y^{k-1}\Vert = \Vert x' - y^{k-1} \Vert \\
x,x' \in E
\end{subarray}}{\sum \cdots \sum} \chi\left( s\left(\Vert x \Vert - \Vert x'\Vert\right)\right)\chi\left(-2sy^k \cdot \left(x-x'\right)\right). 
\end{align*}
It follows that $\mathcal{N} \leq \sum_{\beta} \mathcal{N}_\beta,$ where 
\begin{align*}
\mathcal{N}_{\beta} = n^{-1} \sum_{y^k \in \mathbb{Z}_n^d} \sum_{ \substack{ s \neq 0:\, val(s) = \beta \\ y^1,\dots, y^{k-1} \in E}}  \underset{
\begin{subarray}{c}
\Vert x - y^{1}\Vert = \Vert x' - y^{1} \Vert, \dots, \Vert x - y^{k-1}\Vert = \Vert x' - y^{k-1} \Vert \\
x,x' \in E
\end{subarray}}{\sum \cdots \sum} 
\chi\left( s\left(\Vert x \Vert - \Vert x'\Vert\right)\right)\chi\left(-2sy^k \cdot \left(x-x'\right)\right).  
\end{align*}
Now, we will bound $\mathcal{N}_\beta.$ We proceed similarly as in the initial case. More precisely, applying the orthogonality property (Lemma \ref{ortho-property}), we have 
\begin{align*}
\mathcal{N}_\beta &= n^{-1} \sum_{y^k \in \mathbb{Z}_n^d} \sum_{ \substack{ \overline{s} \in \mathbb{Z}^{\times}_{n'_{\beta}}, \\ y^1,\dots, y^{k-1} \in E}}  \underset{
\begin{subarray}{c}
\Vert x - y^{1}\Vert = \Vert x' - y^{1} \Vert, \dots, \Vert x - y^{k-1}\Vert = \Vert x' - y^{k-1} \Vert \\
x,x' \in E
\end{subarray}}{\sum \cdots \sum} \chi\left( n_{\beta} \overline{s}\left( \Vert x \Vert - \Vert x' \Vert \right)\right)\chi\left( -2n_{\beta} \overline{s} y^k \cdot (x-x')\right)  \\
&= n^{-1} \sum_{y^k \in \mathbb{Z}_n^d} \sum_{ \substack{ \overline{s} \in \mathbb{Z}^{\times}_{n'_{\beta}}, \\ y^1,\dots, y^{k-1} \in E}}  \underset{
\begin{subarray}{c}
\Vert x - y^{1}\Vert = \Vert x' - y^{1} \Vert, \dots, \Vert x - y^{k-1}\Vert = \Vert x' - y^{k-1} \Vert \\
x,x' \in E:\, n_{\beta}(x-x') = \mathbf{0}
\end{subarray}}{\sum \cdots \sum} \chi\left( n_{\beta} \overline{s}\left( \Vert x \Vert - \Vert x' \Vert \right)\right) \\
&= n^{d-1} \sum_{ \substack{ \overline{s} \in \mathbb{Z}^{\times}_{n'_{\beta}}, \\ y^1,\dots, y^{k-1} \in E}}  \underset{
\begin{subarray}{c}
\Vert x - y^{1}\Vert = \Vert x' - y^{1} \Vert, \dots, \Vert x - y^{k-1}\Vert = \Vert x' - y^{k-1} \Vert \\
x,x' \in E:\, n_{\beta}(x-x') = \mathbf{0}
\end{subarray}}{\sum \cdots \sum} \chi\left( n_{\beta} \overline{s}\left( \Vert x \Vert - \Vert x' \Vert \right)\right).
\end{align*} 
Therefore, we obtain 
\begin{align*}
\left|\mathcal{N}_\beta \right| &\leq n^{d-1} \sum_{ \substack{ \overline{s} \in \mathbb{Z}^{\times}_{n'_{\beta}}, \\ y^1,\dots, y^{k-1} \in E}}  \underset{
\begin{subarray}{c}
\Vert x - y^{1}\Vert = \Vert x' - y^{1} \Vert, \dots, \Vert x - y^{k-1}\Vert = \Vert x' - y^{k-1} \Vert \\
x,x' \in E:\, n_{\beta}(x-x') = \mathbf{0}
\end{subarray}}{\sum \cdots \sum} 1 \\
&\leq n^{d-1} n'_\beta|E|^{k-1}  \left| \left\{ x, x' \in E: \, n_\beta (x-x') = \mathbf{0} \right\} \right|. 
\end{align*}
On the other hand, it follows from \eqref{eq.solutions} that
\[ \left| \left\{ x, x' \in E: \, n_\beta (x-x') = \mathbf{0} \right\} \right| \leq n_{\beta}^d|E| = \left( \dfrac{n}{n'_\beta}\right)^d |E|.\]
Hence, we obtain that 
\[ |\mathcal{N}_{\beta}| \leq \dfrac{n^{2d-1}}{n'^{d-1}_{\beta}}|E|^k  \leq \dfrac{n^{2d-1}}{\gamma(n)^{d-1}}|E|^k.\] 
Putting all together, we conclude that 
\begin{align*}
\mathcal{M}_k &\ll \dfrac{|E|^{k+2}}{n^k} + \frac{\tau(n)n^{2d-2}}{\gamma(n)^{d-1}}|E|^k + \mathcal{N} \\ 
&\ll \dfrac{|E|^{k+2}}{n^k} + \frac{\tau(n)n^{2d-2}}{\gamma(n)^{d-1}}|E|^k + \sum_{\beta} \left|\mathcal{N}_{\beta}\right| \\
&\ll \dfrac{|E|^{k+2}}{n^k} + \frac{\tau(n)n^{2d-2}}{\gamma(n)^{d-1}}|E|^k + \frac{\tau(n)n^{2d-1}}{\gamma(n)^{d-1}} |E|^k \\
&\ll \dfrac{|E|^{k+2}}{n^k} + \dfrac{\tau(n)n^{2d-1}}{\gamma(n)^{d-1}}|E|^k.
\end{align*}
This concludes the proof of Lemma \ref{lemma.k-star}.
\end{proof} 

We are now ready to give a proof of Theorem \ref{k-star}.
\begin{proof}[\bf Proof of Theorem \ref{k-star}]
By the Cauchy-Schwarz inequality, we have 
\begin{align*}
|E|^{2k+2} &= \left( \sum_{y^1,\dots,y^k \in E} \sum_{t_1,\dots, t_k \,\in \mathbb{Z}_n} \nu_{y^1,\dots,y^k}(t_1,\dots,t_k)\right)^2 \\
&\leq \sum_{y^1,\dots,y^k \in E} \left| \Delta_{y^1,\dots,y^k}(E)\right| \cdot \sum_{y^1,\dots,y^k \in E} \sum_{t_1,\dots,t_k \in \mathbb{Z}_n}  \nu^2_{y^1,\dots,y^k}(t_1,\dots,t_k).
\end{align*}
It follows from Lemma \ref{lemma.k-star} that 
\[|E|^{2k+2} \ll \sum_{y^1,\dots,y^k \in E} \left|\Delta_{y^1,\dots,y^k}(E)\right| \cdot \left( \dfrac{|E|^{k+2}}{n^k} + \frac{\tau(n)n^{2d-1}}{\gamma(n)^{d-1}} |E|^k\right). \] 
Therefore, we have 
\[ \frac{1}{|E|^k} \sum_{y^1,\dots,y^k} \left|\Delta_{y^1,\dots,y^k}(E)\right| \gg \frac{|E|^{k+2}}{\dfrac{|E|^{k+2}}{n^k}+\frac{\tau(n)n^{2d-1}}{\gamma(n)^{d-1}}|E|^k} \gg n^k\]
under the assumption 
\[ |E| \gg \dfrac{\sqrt{\tau(n)} n^{d+\frac{k-1}{2}}}{\gamma(n)^{(d-1)/2}}.\]
This concludes the proof of Theorem \ref{k-star}.
\end{proof}

\subsection{Distribution of $k$-simplices}
Applying Lemma \ref{lemma.k-star}, we obtain the following result. 
\begin{lemma} \label{lemma.k-simplices}
Given $E \subset \mathbb{Z}_n^d,$ let $X \subset E \times E \times \dots \times E = E^u, \, u \ge 2\,$ with $X \sim |E|^u.$ Define 
\[ X' = \left\{(y^1, \dots, y^{u-1}): (y^1,\dots,y^u) \in X \, \text{ for some } y^u \in E \right\}. \]
For each $(y^1,\dots,y^{u-1}) \in X'$, we define
\[ X(y^1,\dots,y^u) = \left\{ y^u \in E: (y^1,\dots,y^u) \in X^u \right\}.\] 
If 
\[ |E| \gg \dfrac{\sqrt{\tau(n)} n^{d+\frac{u-2}{2}}}{\gamma(n)^{(d-1)/2}}\]
 then 
\[ \frac{1}{|X'|} \sum_{(y^1,\dots,y^{u-1}) \in X'} \left|\Delta_{y^1,\dots,y^{u-1}}\left(X(y^1,\dots,y^{u-1})\right)\right| \gg n^{u-1}, \]
where 
\[ \Delta_{y^1,\dots,y^{u-1}}\left( X(y^1,\dots,y^{u-1})\right) =  \left\{ \left(\Vert y^u - y^1 \Vert, \dots, \Vert y^{u} - y^{u-1} \Vert \right) \in \left(\mathbb{Z}_n\right)^{u-1} \, : \, y^u \in X(y^1,\dots,y^{u-1}) \right\}.\]
\end{lemma}
\begin{proof}
For each $(t_1, \dots, t_{u-1}) \in \left(\mathbb{Z}_n\right)^{u-1},$ define the incidence function on $X(y^1,\dots,y^{u-1})$ as follows
\[ \nu_{y^1,\dots,y^{u-1}}^{X(y^1,\dots,y^{u-1})}(t_1,\dots,t_{u-1}) =\left| \left\{ y^u \in X(y^1,\dots,y^{u-1}) \, : \,\Vert y^u - y^1 \Vert = t_1, \dots, \Vert y^u - y^{u-1} \Vert = t_{u-1} \right\} \right|.    \] 
It is easy to see that 
\[ \nu_{y^1,\dots,y^{u-1}}^{X(y^1,\dots,y^{u-1})}(t_1,\dots,t_{u-1})  \leq \nu_{y^1,\dots,y^{u-1}}(t_1,\dots,t_u),\]
where 
\[ \nu_{y^1,\dots,y^{u-1}}(t_1,\dots,t_{u-1}) = \left| \left\{ y^u \in E \, : \, \Vert y^u - y^1 \Vert = t_1, \dots, \Vert y^u - y^{u-1} \Vert = t_{u-1} \right\} \right|.\]
By the Cauchy-Schwarz inequality, we have 
\begin{align*}
|E|^2 &= \left( \sum_{(y^1,\dots,y^{u-1}) \in X'} \sum_{t_1,\dots,t_{u-1} \in \mathbb{Z}_n} \nu_{y^1,\dots,y^{u-1}}^{X(y^1,\dots,y^{u-1})}(t_1,\dots,t_{u-1})\right)^2 \\ 
&\leq \left( \sum_{(y^1,\dots,y^{u-1}) \in X'} \left|\Delta_{y^1,\dots,y^{u-1}}\left( X(y^1,\dots,y^{u-1})\right)\right|\right)\left(\sum_{(y^1,\dots,y^{u-1}) \in E} \sum_{t_1,\dots,t_{u-1} \in \mathbb{Z}_n} \nu^2_{y^1,\dots,y^{u-1}}(t_1,\dots,t_u)\right).
\end{align*}
Using Lemma \ref{lemma.k-star}, we have 
\[ |E|^2 \leq \left( \sum_{(y^1,\dots,y^{u-1}) \in X'} \left|\Delta_{y^1,\dots,y^{u-1}}\left( X(y^1,\dots,y^{u-1})\right)\right|\right)\cdot \left( \dfrac{|E|^{u+1}}{n^{u-1}} + \frac{\tau(n)n^{2d-1}}{\gamma(n)^{d-1}}|E|^{u-1}\right). \]
On the other hand, since $X' \sim |E|^{u-1},$ we have 
\[ \frac{1}{|X'|}\sum_{(y^1,\dots,y^{u-1}) \in X'} \left|\Delta_{y^1,\dots,y^{u-1}}\left( X(y^1,\dots,y^{u-1})\right)\right| \gg \dfrac{|E|^{u+1}}{\frac{|E|^{u+1}}{n^{u-1}} + \frac{\tau(n)n^{2d-1}}{\gamma(n)^{d-1}}|E|^{u-1}} \gg n^{u-1}\]
under the assumption 
\[ |E| \gg \dfrac{\sqrt{\tau(n)} n^{d+\frac{u-2}{2}}}{\gamma(n)^{(d-1)/2}}.\]
This concludes the proof of Lemma \ref{lemma.k-simplices}.  
\end{proof}

As a direct consequence, we have the following corollary.
\begin{corollary}\label{lem.k-simplices}
Let $E \subset \mathbb{Z}_n^d$ and $X \subset E \times \dots \times E = E^u, \, u \ge 2,$ with $|X| \sim |E|^u.$ If 
\[|E| \gg \dfrac{\sqrt{\tau(n)} n^{d+\frac{u-2}{2}}}{\gamma(n)^{(d-1)/2}}, \]
then there exists $\mathcal{X}^{(1)} \subset X' \subset E^{u-1}$ with $|\mathcal{X}^{(1)}| \sim |X'| \sim |E|^{u-1}$ such that for every $(y^1,\dots,y^{u-1}) \in \mathcal{X}^{(1)},$ we have 
	\begin{equation*}
	 \left| \Delta_{y^1,\dots,y^{u-1}}\left( X(y^1,\dots,y^{u-1}\right)\right| \gg n^{u-1}.
	\end{equation*}
Namely, the elements in $X$ determine a positive proportion of all $(u-1)$-simplices which are based on a $(u-2)-$simplex given by any element $(y^1,\dots,y^{u-1}) \in \mathcal{X}^{(1)}.$
\end{corollary}

\begin{proof}[\bf Proof of Theorem \ref{main.k-simplices}]
Firstly, by Theorem \ref{k-star}, there exists a subset $\mathcal{X}^{(0)} \subset E\times \dots \times E = E^k$ with $|\mathcal{X}^{(0)}| \sim |E|^k$ such that for every $(y^1,\dots,y^k) \in \mathcal{X}^{(0)},$ we have 
\[ \left| \Delta_{y^1,\dots,y^k}(E)\right| = \left| \left\{ \left( \Vert y^0 - y^1 \Vert, \dots, \Vert y^0 - y^k \Vert \right) \in \left(\mathbb{Z}_n\right)^k\, : \, y^0 \in E \right\} \right| \gg n^k.\]
This implies that the set $E$ determines a positive proportion of all $k$-simplices which are based on a $(k-1)-$simplex given by any element $(y^1,\dots,y^{k}) \in \mathcal{X}^{(0)}.$

Since 
\[ |E| \gg \frac{\sqrt{\tau(n)}n^{d+\frac{k-1}{2}}}{\gamma(n)^{(d-1)/2}} \gg \frac{\sqrt{\tau(n)}n^{d+\frac{k-2}{2}}}{\gamma(n)^{(d-1)/2}}\]
and $\left|\mathcal{X}^{(0)}\right| \sim |E|^k,$ by Corollary \ref{lem.k-simplices} where $u$ is replaced by $k$, there exists a set $\mathcal{X}^{(1)} \subset \left(\mathcal{X}^{(0)}\right)' \subset E^{k-1}$ with $\left|\mathcal{X}^{(1)}\right| \sim \left|\left(\mathcal{X}^{(0)}\right)'\right| \sim |E|^{k-1}$ such that for every $(y^1,\dots,y^{k-1}) \in \mathcal{X}^{(1)},$ we have
\[ \left| \Delta_{y^1,\dots,y^{k-1}}\left(\mathcal{X}^{(0)}(y^1,\dots,y^{k-1})\right) \right| \gg n^{k-1}.\] 
This implies that the set $\mathcal{X}^{(0)}$ determines a positive proportion of all $(k-1)$-simplices which are based on a $(k-2)-$simplex given by any element $(y^1,\dots,y^{k-1}) \in \mathcal{X}^{(1)}.$

Again, applying Corollary \ref{lem.k-simplices} where $u$ is replaced by $(k-1)$, there exists a set $\mathcal{X}^{(2)} \subset \left(\mathcal{X}^{(1)}\right)' \subset E^{k-2}$ with $\left|\mathcal{X}^{(2)}\right| \sim \left|\left(\mathcal{X}^{(1)}\right)'\right| \sim |E|^{k-2}$ such that for every $(y^1,\dots,y^{k-2}) \in \mathcal{X}^{(2)},$ we have 
\[ \left| \Delta_{y^1,\dots,y^{k-2}}\left(\mathcal{X}^{(1)}(y^1,\dots,y^{k-2})\right) \right| \gg n^{k-2}.\] 
This implies that the set $\mathcal{X}^{(1)}$ determines a positive proportion of all $(k-2)$-simplices which are based on a $(k-3)-$simplex given by any element $(y^1,\dots,y^{k-2}) \in \mathcal{X}^{(2)}.$ 

Repeating the above process, there exists a sequence of sets $\mathcal{X}^{(0)},\mathcal{X}^{(1)},\dots,\mathcal{X}^{(k-2)}$ with $\left|\mathcal{X}^{(s)}\right| = |E|^{k-s}$ for all $s = 0, 1, \dots, k-2$ such that the set $\mathcal{X}^{(s)}$ determines a positive proportion of all $(k-1-s)-$simplices which are based on a $(k-2-s)-$simplex given by any element $(y^1,\dots,y^{k-1-s}) \in \mathcal{X}^{(s+1)}.$ 

Finally, let $u = 2, X = \mathcal{X}^{(k-2)}$. Applying Lemma \ref{lem.k-simplices}, we have  the set $\mathcal{X}^{(k-2)} \subset E \times E$ determines a positive proportion of all $1-$simplices. This implies that the set $\mathcal{X}^{(0)}$ determines a positive proportion of all $(k-1)$-simplices.  

On the other hand, since the set $E$ determines a positive proportion of all $k$-simplices whose bases are fixed as a $(k-1)-$simplex given by any element $(y^1,\dots,y^{k}) \in \mathcal{X}^{(0)},$ we conclude that the set $E$ determines a positive proportion of all $k-$simplices. It means that
\[ \left|\mathcal{T}_k(E)\right| \ge n^{\binom{k+1}{2}},\]
concluding the proof of Theorem \ref{main.k-simplices}. 
\end{proof}

\section{Dot-product simplices - Proof of Theorem }

\subsection{Counting dot-product stars}

Define dot-product $k$-star set determined by $k$ points $y^1,\dots,y^k \in E$ as follows 
\[ \Pi_{y^1,y^2,\ldots,y^k}(E) = \left\{\left(x \cdot y^1,\ldots, x \cdot y^k\right) \in \mathbb{Z}_q^k: x \in E \right\}.\]
The main result of this section is to count the number of dot-product $k$-stars with bases in a point set $E$. 
\begin{theorem}\label{k-star dot product}
Let $E \subset \mathbb{Z}_n^d$ with $n \ge 3$ be an odd integer. Suppose that 
\[ |E| \gg \dfrac{\sqrt{\tau(n)}n^{d+\frac{k-1}{2}}}{\gamma(n)^{(d-1)/2}}.\]
Then, we have 
\[ \frac{1}{|E|^k} \sum_{y^1, \dots , y^k \in E} \left| \Pi_{y^1,\dots,y^k}(E) \right| \gg n^{k}.\]
\end{theorem}

For $t_1,\dots,t_k \in \mathbb{Z}_n$ and $E \subset \mathbb{Z}_n^d,$ we define the counting function
\begin{align*}
 \mu_{y^1,\dots,y^k}(t_1,\dots,t_k) &:= \left|\left\{ x \in E \,:\, x \cdot y^i = t_i,\, \forall i = 1, \dots ,k \right\}\right|\\
&= \sum_{x \in E} \prod_{i=1}^k \left(n^{-1} \sum_{s \in \mathbb{Z}_n} \chi\left(s(t_i - x \cdot y^i)\right)\right).
\end{align*}
The following lemma plays an significant role in the proof of Theorem \ref{k-star dot product}.
\begin{lemma}\label{lemma-k-dot.product} Let $E \subset \mathbb{Z}_n^d$ with odd integer $n \ge 3.$ Then 
\[ \mathcal{K}_{k} = \sum_{y^1,\dots,y^k \in E} \sum_{t_1,\dots,t_k \in \mathbb{Z}_n} \mu^2_{y^1,\dots,y^k}(t_1,\dots,t_k) \ll \dfrac{|E|^{k+2}}{n^k} + \frac{\tau(n)n^{2d-1}}{\gamma(n)^{d-1}}|E|^k.\]
\end{lemma}
\begin{proof}
We proceed by induction on $k$. For the initial case $k = 1,$ we use the notation $\mu_y(t)$ instead of $\mu_{y^1}(t_1).$ More precisely, define the counting function 
\[ \mu_y(t) := \left| \left\{ x \in E\, : \, x \cdot y = t \right\}\right|.\] 
Applying the orthogonality property, we have
\[ \mu_y(t) = \sum_{x \in E} n^{-1} \sum_{s \in \mathbb{Z}_n}\chi\left(s(t-x\cdot y)\right) = n^{-1}\sum_{x \in \mathbb{Z}_n^d}\sum_{s \in \mathbb{Z}_n}\chi\left(s(t-x\cdot y)\right)E(x).\]  
It follows that  
\begin{align*}
\widehat{\mu_y}(s) &= n^{-1} \sum_{t \in \mathbb{Z}_n} \chi\left(-ts\right) \mu_y(t)  = n^{-2} \sum_{t \in \mathbb{Z}_n} \chi(-ts)\sum_{x \in \mathbb{Z}_n^d}\sum_{s' \in \mathbb{Z}_n}\chi\left(s'(t-x\cdot y)\right)E(x) \\
&= n^{-2} \sum_{x \in \mathbb{Z}_n^d}E(x)\sum_{s' \in \mathbb{Z}_n}\chi\left(-s' x \cdot y\right)\sum_{t \in \mathbb{Z}_n} \chi\left(t(s'-s)\right) \\
&= n^{-1} \sum_{x \in \mathbb{Z}_n^d} E(x) \chi(-s x \cdot y).   \quad \quad (\text{If } s' \neq s, \text{the sum is vanished by Lemma \ref{ortho-property}}.)
\end{align*}
Therefore, we obtain $\widehat{\mu_y}(s) = n^{d-1} \widehat{E}(sy).$ Hence, 
\begin{align*}
\sum_{y \in E} \sum_{s \in \mathbb{Z}_n} \left|\widehat{\mu_y}(s) \right|^2 = n^{2(d-1)} \sum_{y \in E} \sum_{s \in \mathbb{Z}_n} \left|\widehat{E}(sy)\right|^2 = q^{-2}|E|^3 + \mathcal{K}
\end{align*} 
where $\mathcal{K} = n^{2(d-1)}\sum_{s \neq 0} \sum_{y \in E}  \left|\widehat{E}(sy)\right|^2.$ 

Without loss of generality, we suppose that $n$ has the prime decomposition $n = p_1^{\alpha_1}p_2^{\alpha_2}\dots p_{\ell}^{\alpha_{\ell}},$ where $2 < p_1 < p_2 < \dots < p_k$ and $\alpha_i > 0$ for each $i = 1,2,\dots,\ell.$ Define 
\[ val(s) := (val_{p_1}(s),\dots,val_{p_\ell}(s))\]
where $val_{p_i}(x) = r$ if $p_i^r | x$ but $p_i^{r+1} \nmid x.$ 
For each $s \neq 0,$ we write $s = p_1^{\beta_1}\dots p_\ell^{\beta_\ell} \overline{s}$ where $\overline{s} \in \mathbb{Z}_{n'}^{\times}$ is uniquely determined for $n'=p_1^{\alpha_1 - \beta_1} \dots p_k^{\alpha_\ell - \beta_\ell}$ and $\beta_i \ge 0.$ Since $s \neq 0$, $\beta_i < \alpha_i$ for some $i$. We will use the notation $\sum_{\beta}$ to denote the sum over all such $(\beta_1,\dots,\beta_\ell)$'s. Now, rewrite $\mathcal{K}$ as $\mathcal{K} = \sum_{\beta} \mathcal{K}_\beta$, where 
\[ \mathcal{K}_\beta = n^{2(d-1)} \sum_{s \in \mathbb{Z}_n:\, val(s) = \beta} \sum_{y \in E} \left|\widehat{E}(sy)\right|^2.\]
For each $\beta = (\beta_1, \dots,\beta_{\ell}),$ denote $n'_\beta = p_1^{\alpha_1 - \beta_1} \dots p_{\ell}^{\alpha_\ell - \beta_\ell}$ and $n_\beta = p_1^{\beta_1},\dots,p_{\ell}^{\beta_\ell}.$ Now, we will bound $\mathcal{K}_\beta.$ Applying the orthogonality property (Lemma \ref{ortho-property}), we have  
\begin{align*}
\mathcal{K}_\beta &= n^{2(d-1)} \sum_{\overline{s} \in \mathbb{Z}^{\times}_{n'_\beta}} \sum_{y \in E} \left|\widehat{E}\left(p_1^{\beta_1}\dots p_{\ell}^{\beta_\ell} \overline{s} y\right)\right|^2 \\ 
&= n^{2(d-1)} \sum_{\overline{s} \in \mathbb{Z}^{\times}_{n'_\beta}} \sum_{y \in \mathbb{Z}^d_n} E\left(y/\overline{s}\right)\left|\widehat{E}\left(p_1^{\beta_1}\dots p_{\ell}^{\beta_\ell}y\right)\right|^2. 
\end{align*}
Set $\rho(x)=\left|\left\{ y \in \mathbb{Z}_n^d\,:\, p_1^{\beta_1}\dots p_{\ell}^{\beta_\ell}y = x\right\}\right|$. Since $\sum_{\overline{s} \in \mathbb{Z}^{\times}_{n'_\beta}} E(y/\overline{s}) \leq n'_\beta,$ we obtain 
\begin{align*}
\mathcal{K}_\beta &\leq n^{2(d-1)} n'_\beta \sum_{y \in \mathbb{Z}_n^d} \left|\widehat{E}\left(p_1^{\beta_1}\dots p_{\ell}^{\beta_\ell}y\right)\right|^2 \\
&\leq n'_\beta n^{2(d-1)} \sum_{x \in \mathbb{Z}^d_n} \rho(x)\left|\widehat{E}(x)\right|^2  \\
&\leq \left(\max_{x \in \mathbb{Z}_n^d} \rho(x)\right) n'_\beta n^{2(d-1)}  \sum_{x \in \mathbb{Z}^d_n}\left|\widehat{E}(x)\right|^2 \\
 &= \left(\max_{x \in \mathbb{Z}_n^d} \rho(x)\right) n'_\beta n^{d-2}|E|,
\end{align*} 
where the last line follows by \eqref{Plancherel-Eq}.
On the other hand, similarly to the proof of \eqref{eq.solutions}, it is not hard to show that 
\[\rho(x)=\left|\left\{ y \in \mathbb{Z}_n^d\,:\, p_1^{\beta_1}\dots p_{\ell}^{\beta_\ell}y = x\right\}\right| \leq \left(p_1^{\beta_1}\dots p_{\ell}^{\beta_\ell}\right)^{d} = n_\beta^d.\] 
It implies that  
\[ \mathcal{K}_\beta \leq n_\beta^d n'_\beta n^{d-2} |E| = \dfrac{n^{2d-2}|E|}{n'^{d-1}_{\beta}} \leq \dfrac{n^{2d-2}|E|}{\gamma(n)^{d-1}}.\] 
Therefore, applying Plancherel identity \eqref{Plancherel-Eq} again, we have 
\[ \sum_{ t \in \mathbb{Z}_n} \sum_{y \in E} \mu^2_y(t) = n \sum_{s \in \mathbb{Z}_n} \sum_{y \in E} \left|\widehat{\mu_y}(s)\right|^2 \leq  n^{-1}|E|^3 + \dfrac{n^{2d-1}|E|}{\gamma(n)^{d-1}}.\]
This concludes the proof for the initial case $k = 1$ of Lemma \ref{lemma-k-dot.product}. 

Now, suppose that the statement holds for $k-1$
\[ \mathcal{K}_{k-1} = \sum_{y^1,\dots,y^{k-1} \in E} \sum_{t_1,\dots,t_{k-1} \in \mathbb{Z}_n} \mu^2_{y^1,\dots,y^{k-1}}(t_1,\dots,t_{k-1}) \ll \dfrac{|E|^{k+1}}{n^{k-1}} + \frac{\tau(n)n^{2d-1}}{\gamma(n)^{d-1}}|E|^{k-1}.\] 
We will show that the statement holds for $k$.
Firstly, set $\mathbf{s} = (s_1,\dots,s_k) \in \mathbb{Z}_n^k,$ we have 
\begin{align*}
\widehat{\mu}_{y^1,\dots,y^k}(s_1,\dots,s_k) &= n^{-k} \sum_{t_1,\dots,t_k \in \mathbb{Z}_n} \chi\left(-t_1s_1- \dots - t_ks_k\right) \mu_{y^1,\dots,y^k}(t_1,\dots,t_k) \\
&= n^{-k} \sum_{t_1,\dots,t_k \in \mathbb{Z}_n} \chi\left(-t_1s_1- \dots - t_ks_k\right) \sum_{x \in E} \prod_{i=1}^k \left(n^{-1} \sum_{s' \in \mathbb{Z}_n} \chi\left(s'(t_i - x \cdot y^i)\right)\right) \\
&= n^{-2k}\sum_{\mathbf{t} = (t_1,\dots,t_k) \in \mathbb{Z}^k_n}\chi\left(-\mathbf{t}\cdot \mathbf{s}\right) \sum_{x \in E} \sum_{\mathbf{s'}=(s'_1,\dots,s'_k)\in \mathbb{Z}_n^k} \prod_{i=1}^k \chi\left(s'_i t_i - s'_i x\cdot y^i\right) \\
&= n^{-2k}\sum_{x \in E}\sum_{\mathbf{s'}=(s'_1,\dots,s'_k)\in \mathbb{Z}_n^k} \sum_{\mathbf{t} = (t_1,\dots,t_k) \in \mathbb{Z}^k_n}\chi\left(\mathbf{t}\cdot \left(\mathbf{s'}-\mathbf{s}\right)\right)\chi\left(-x\cdot (s'_1y^1+\dots+s'_ky^k)\right).
\end{align*}
Applying the orthogonality property, we have  
\begin{align*}
\sum_{\mathbf{s'} \in \mathbb{Z}_n^k \, : \, \mathbf{s'} \neq \mathbf{s}} \sum_{\mathbf{t} \in \mathbb{Z}_n^k} \chi\left( \mathbf{t} \cdot (\mathbf{s'}-\mathbf{s})\right) = 0.  
\end{align*}
Therefore, we obtain 
\begin{align*}
\widehat{\mu}_{y^1,\dots,y^k}(s_1,\dots,s_k) &= n^{-k} \sum_{x \in E} \chi\left(-x \cdot \left(s_1y^1+\dots+s_ky^k\right)\right) \\
&= n^{d-k} \widehat{E}\left(s_1y^1+\dots+s_ky^k\right). 
\end{align*}
It follows that 
\[\sum_{y^1,\dots,y^k \in E} \sum_{s_1,\dots,s_k \in \mathbb{Z}_n} \left|\widehat{\mu}_{y^1,\dots,y^k}(s_1,\dots,s_k)\right|^2 = n^{2(d-k)}\sum_{y^1,\dots,y^k \in E} \sum_{s_1,\dots,s_k \in \mathbb{Z}_n} \left|\widehat{E}\left(s_1y^1+\dots+s_ky^k\right)\right|^2. \]
Separating the case $s_k \neq 0$, we have 
\[\sum_{y^1,\dots,y^k \in E} \sum_{s_1,\dots,s_k \in \mathbb{Z}_n} \left|\widehat{\mu}_{y^1,\dots,y^k}(s_1,\dots,s_k)\right|^2 = I + II,\]
where 
\begin{align*}
I &= n^{2(d-k)}\sum_{\substack{y^1,\dots,y^{k-1} \in E \\ y^k \in E}} \sum_{s_1,\dots,s_{k-1} \in \mathbb{Z}_n} \left|\widehat{E}\left(s_1y^1+\dots+s_{k-1}y^{k-1}\right)\right|^2,  \\
II &= n^{2(d-k)}\sum_{s_k \neq 0}\sum_{y^1,\dots,y^k \in E} \sum_{s_1,\dots,s_{k-1} \in \mathbb{Z}_n} \left|\widehat{E}\left(s_1y^1+\dots+s_ky^k\right)\right|^2.
\end{align*}
Using the above estimation, applying Plancherel identity and the induction hypothesis, we have
\begin{align*}
I &= n^{2(d-k)}|E|\sum_{y^1,\dots,y^{k-1} \in E} \sum_{s_1,\dots,s_{k-1} \in \mathbb{Z}_n} \left|\widehat{E}\left(s_1y^1+\dots+s_{k-1}y^{k-1}\right)\right|^2 \\
&= |E| \sum_{y^1,\dots,y^{k-1} \in E} \sum_{s_1,\dots,s_{k-1} \in \mathbb{Z}_n} \left|\widehat{\mu}_{y^1,\dots,y^{k-1}}(s_1,\dots,s_{k-1})\right|^2  \\
&= n^{-k-1}|E| \mathcal{K}_{k-1} \\
&\ll  \dfrac{|E|^{k+2}}{n^{2k}} + \frac{\tau(n)n^{2d-k-2}}{\gamma(n)^{d-1}}|E|^{k}.
\end{align*}
Now, we will bound the second term $II.$ Note that $II = \sum_{\beta} II_\beta$ where 
\[II_\beta = n^{2(d-k)}\sum_{\substack{ s_k \neq 0: \\ val(s_k) = \beta} }\sum_{y^1,\dots,y^k \in E} \sum_{s_1,\dots,s_{k-1} \in \mathbb{Z}_n} \left|\widehat{E}\left(s_1y^1+\dots+s_ky^k\right)\right|^2.  \] 
We proceed similar to the initial case. We have
\begin{align*} 
II_\beta &= n^{2(d-k)} \sum_{y^1,\dots,y^{k-1} \in E} \sum_{s_1,\dots,s_{k-1} \in \mathbb{Z}_n} \left(\sum_{\overline{s}_k \in \mathbb{Z}^{\times}_{n'_\beta}}\sum_{y^k \in \mathbb{Z}_n^d} E(y^k) \left| \widehat{E}\left(s_1y^1+\dots + s_{k-1}y^{k-1} + n_\beta \overline{s}_k y^k \right) \right|^2\right) \\
&= n^{2(d-k)} \sum_{y^1,\dots,y^{k-1} \in E} \sum_{s_1,\dots,s_{k-1} \in \mathbb{Z}_n}\left(\sum_{\overline{s}_k \in \mathbb{Z}^{\times}_{n'_\beta}}\sum_{y^k \in \mathbb{Z}_n^d} E(y^k/\overline{s}_k) \left| \widehat{E}\left(s_1y^1+\dots + s_{k-1}y^{k-1} + n_\beta y^k \right) \right|^2\right) \\
&< n'_\beta n^{2(d-k)}\sum_{y^1,\dots,y^{k-1} \in E} \sum_{s_1,\dots,s_{k-1} \in \mathbb{Z}_n}  \sum_{y^k \in \mathbb{Z}_n^d} \left|\widehat{E}(s_1y^1+\dots+s_{k-1}y^{k-1} + n_\beta y^k )\right|^2 \\ 
&\leq n'_\beta n^{2(d-k)}\sum_{y^1,\dots,y^{k-1} \in E} \sum_{s_1,\dots,s_{k-1} \in \mathbb{Z}_n}  \sum_{x \in \mathbb{Z}_n^d} \rho\left(x-s_1y^1-\dots - s_{k-1}y^{k-1}\right) \left|\widehat{E}(x)\right|^2 \\
&\leq \left(\max_{x \in \mathbb{Z}_n^d} \rho(x)\right) n'_\beta n^{2d-k-1} |E|^{k-1} \sum_{x \in \mathbb{Z}_n^d} \left|\widehat{E}(x)\right|^2.
\end{align*}
Using $\rho(x) \leq n_\beta^d$ and applying Plancherel identity, we obtain 
\begin{align*}
II_\beta \leq n^d_\beta n'_\beta n^{2d-k-1} |E|^{k-1} \left( n^{-d}|E|\right) = \dfrac{n^{2d-k-1}|E|^{k}}{n'^{d-1}_\beta} \leq \dfrac{n^{2d-k-1}}{\gamma(n)^{d-1}} |E|^k. 
\end{align*}
Therefore, we have  
\[ II = \sum_{\beta} II_{\beta} \leq \dfrac{\tau(n)n^{2d-k-1}}{\gamma(n)^{d-1}} |E|^k.\] 
Finally, applying Plancherel identity, we have 
\begin{align*}
 \mathcal{K}_k &= n^{k} \left( I + II \right) \\
 &\ll n^{k} \left( \dfrac{|E|^{k+2}}{n^{2k}} + \dfrac{\tau(n)n^{2d-k-2}}{\gamma(n)^{d-1}}|E|^k + \dfrac{\tau(n) n^{2d-k-1}}{\gamma(n)^{d-1}}|E|^k\right) \\ 
 &\ll \dfrac{|E|^{k+2}}{n^{k}} + \dfrac{\tau(n) n^{2d-1}}{\gamma(n)^{d-1}}|E|^k.
\end{align*} 
This concludes the proof of Lemma \ref{lemma-k-dot.product}. 
\end{proof}

We are now ready to give a proof of Theorem \ref{k-star dot product}.

\begin{proof}[\bf Proof of Theorem \ref{k-star dot product}]
Applying the Cauchy-Schwarz inequality, we have 
\begin{align*}
|E|^{2k+2} &= \left( \sum_{y^1,\dots,y^k \in E} \sum_{t_1,\dots, t_k \,\in \mathbb{Z}_n} \mu_{y^1,\dots,y^k}(t_1,\dots,t_k)\right)^2 \\
&\leq \sum_{y^1,\dots,y^k \in E} \left| \Pi_{y^1,\dots,y^k}(E)\right| \cdot \sum_{y^1,\dots,y^k \in E} \sum_{t_1,\dots,t_k \in \mathbb{Z}_n}  \mu^2_{y^1,\dots,y^k}(t_1,\dots,t_k).
\end{align*}
It follows from Lemma \ref{lemma-k-dot.product} that 
\[|E|^{2k+2} \ll \sum_{y^1,\dots,y^k \in E} \left|\Pi_{y^1,\dots,y^k}(E)\right| \cdot \left( \dfrac{|E|^{k+2}}{n^k} + \frac{\tau(n)n^{2d-1}}{\gamma(n)^{d-1}} |E|^k\right). \] 
Therefore, we have 
\[ \frac{1}{|E|^k} \sum_{y^1,\dots,y^k} \left|\Pi_{y^1,\dots,y^k}(E)\right| \gg \frac{|E|^{k+2}}{\dfrac{|E|^{k+2}}{n^k}+\frac{\tau(n)n^{2d-1}}{\gamma(n)^{d-1}}|E|^k} \gg n^k\]
under the assumption 
\[ |E| \gg \dfrac{\sqrt{\tau(n)} n^{d+\frac{k-1}{2}}}{\gamma(n)^{(d-1)/2}}.\]
This concludes the proof of Theorem \ref{k-star dot product}.
\end{proof}

\subsection{Distribution of dot-product simplices}

If $k = 1,$ Theorem \ref{main.k-simplices dot product} follows directly from Theorem \ref{k-star dot product}. We only need to consider the case $k \ge 2.$ We will need the following generalization of Theorem \ref{k-star dot product}.
\begin{lemma} \label{lemma.k-simplices with dot-product}
Given $E \subset \mathbb{Z}_n^d,$ let $Y \subset E \times E \times \dots \times E = E^u, \, u \ge 2\,$ with $Y \sim |E|^u.$ Define 
\[ Y' = \left\{(y^1, \dots, y^{u-1}): (y^1,\dots,y^u) \in Y \, \text{ for some } y^u \in E \right\}. \]
For each $(y^1,\dots,y^{u-1}) \in Y'$, we define
\[ Y(y^1,\dots,y^u) = \left\{ y^u \in E: (y^1,\dots,y^u) \in Y^u \right\}.\] 
If 
\[ |E| \gg \dfrac{\sqrt{\tau(n)} n^{d+\frac{u-2}{2}}}{\gamma(n)^{(d-1)/2}},\]
 then we have
\[ \frac{1}{|Y'|} \sum_{(y^1,\dots,y^{u-1}) \in Y'} \left|\Pi_{y^1,\dots,y^{u-1}}\left(Y(y^1,\dots,y^{u-1})\right)\right| \gg n^{u-1}, \]
where 
\[ \Pi_{y^1,\dots,y^{u-1}}\left( Y(y^1,\dots,y^{u-1})\right) =  \left\{ \left(y^u \cdot y^1 , \dots,  y^{u} \cdot y^{u-1} \right) \in \left(\mathbb{Z}_n\right)^{u-1} \, : \, y^u \in Y(y^1,\dots,y^{u-1}) \right\}.\]
\end{lemma}
\begin{proof}
For each $(t_1, \dots, t_{u-1}) \in \left(\mathbb{Z}_n\right)^{u-1},$ define the incidence function on $Y(y^1,\dots,y^{u-1})$ as follows 
\[ \mu_{y^1,\dots,y^{u-1}}^{Y(y^1,\dots,y^{u-1})}(t_1,\dots,t_{u-1}) =\left| \left\{ y^u \in Y(y^1,\dots,y^{u-1}) \, : \,  y^u \cdot y^1 = t_1, \dots, y^u \cdot y^{u-1} = t_{u-1}  \right\} \right|.\] 
It is easy to see that
\[ \mu_{y^1,\dots,y^{u-1}}^{Y(y^1,\dots,y^{u-1})}(t_1,\dots,t_{u-1})  \leq \mu_{y^1,\dots,y^{u-1}}(t_1,\dots,t_u),\]
where 
\[ \mu_{y^1,\dots,y^{u-1}}(t_1,\dots,t_{u-1}) = \left| \left\{ y^u \in E \, : \,  y^u \cdot y^1  = t_1, \dots, y^u \cdot y^{u-1}  = t_{u-1}  \right\} \right|.\]
By the Cauchy-Schwarz inequality, we have 
\begin{align*}
|E|^2 &= \left( \sum_{(y^1,\dots,y^{u-1}) \in Y'} \sum_{t_1,\dots,t_{u-1} \in \mathbb{Z}_n} \mu_{y^1,\dots,y^{u-1}}^{Y(y^1,\dots,y^{u-1})}(t_1,\dots,t_{u-1})\right)^2 \\ 
&\leq \left( \sum_{(y^1,\dots,y^{u-1}) \in Y'} \left|\Pi_{y^1,\dots,y^{u-1}}\left( Y(y^1,\dots,y^{u-1})\right)\right|\right)\cdot \left(\sum_{(y^1,\dots,y^{u-1}) \in E} \sum_{t_1,\dots,t_{u-1} \in \mathbb{Z}_n} \mu^2_{y^1,\dots,y^{u-1}}(t_1,\dots,t_u)\right).
\end{align*}
Using Lemma \ref{lemma-k-dot.product}, we obtain
\[ |E|^2 \leq \left( \sum_{(y^1,\dots,y^{u-1}) \in Y'} \left|\Pi_{y^1,\dots,y^{u-1}}\left( Y(y^1,\dots,y^{u-1})\right)\right|\right)\cdot \left( \dfrac{|E|^{u+1}}{n^{u-1}} + \frac{\tau(n)n^{2d-1}}{\gamma(n)^{d-1}}|E|^{u-1}\right). \]
On the other hand, since $Y' \sim |E|^{u-1},$ we have
\[ \frac{1}{|Y'|}\sum_{(y^1,\dots,y^{u-1}) \in Y'} \left|\Pi_{y^1,\dots,y^{u-1}}\left( Y(y^1,\dots,y^{u-1})\right)\right| \gg \dfrac{|E|^{u+1}}{\frac{|E|^{u+1}}{n^{u-1}} + \frac{\tau(n)n^{2d-1}}{\gamma(n)^{d-1}}|E|^{u-1}} \gg n^{u-1}\]
under the assumption 
\[ |E| \gg \dfrac{\sqrt{\tau(n)} n^{d+\frac{u-2}{2}}}{\gamma(n)^{(d-1)/2}}.\]
This completes the proof of Lemma \ref{lemma.k-simplices with dot-product}.   
\end{proof}

As a direct consequence, we have the following corollary.
\begin{corollary}\label{coro-k-dot-product}
Let $E \subset \mathbb{Z}_n^d$ and $Y \subset E \times \dots \times E = E^u, \, u \ge 2,$ with $|Y| \sim |E|^u.$ If 
\[|E| \gg \dfrac{\sqrt{\tau(n)} n^{d+\frac{u-2}{2}}}{\gamma(n)^{(d-1)/2}}, \]
then there exists $\mathcal{Y}^{(1)} \subset Y' \subset E^{u-1}$ with $|\mathcal{Y}^{(1)}| \sim |Y'| \sim |E|^{u-1}$ such that for every $(y^1,\dots,y^{u-1}) \in \mathcal{Y}^{(1)},$ we have 
	\begin{equation*}
	 \left| \Pi_{y^1,\dots,y^{u-1}}\left( Y(y^1,\dots,y^{u-1}\right)\right| \gg n^{u-1}.
	\end{equation*}
Namely, the set $Y$ determines a positive proportion of all dot-product $(u-1)$-simplices which are based on a $(u-2)-$simplex given by any element $(y^1,\dots,y^{u-1}) \in \mathcal{Y}^{(1)}.$
\end{corollary}

\begin{proof}[\bf Proof of Theorem \ref{main.k-simplices dot product}]
The proof of Theorem \ref{main.k-simplices dot product} is similary to the proof of Theorem \ref{main.k-simplices}, in which we use Theorem \ref{k-star dot product} and Corollary \ref{coro-k-dot-product} instead of Theorem \ref{k-star} and Corollary \ref{lem.k-simplices}.
\end{proof}

\bibliographystyle{amsplain}

\end{document}